\documentclass{amsart}

\usepackage{subcaption}
\usepackage{color}
\usepackage{xparse}
\usepackage{enumitem}
\usepackage{needspace}
\usepackage{graphicx,amsmath,amsthm,enumitem,hyperref, amssymb, mathtools, subcaption, theoremref}

\newtheorem{theorem}{Theorem}
\newtheorem{lemma}[theorem]{Lemma}
\newtheorem{corollary}[theorem]{Corollary}
\newtheorem{remark}[theorem]{Remark}

\DeclareMathOperator\supp{supp}
\newcommand{\ra}{\rightarrow}
\newcommand{\E}{\mathbf{E}}
\renewcommand{\P}{\mathbf{P}}
\renewcommand{\emptyset}{\varnothing}
\renewcommand{\b}{b}

\newcommand{\Shield}{\mathsf{Shield}}
\NewDocumentCommand{\SurvFirst}{o}{%
	\mathsf{Surv}_{v_0 = \IfValueTF{#1}{#1}{v}}%
}

\title{The critical velocity of the bullet process appears pathwise}
\author{Josh Meisel}\address{Josh Meisel, Department of Mathematics, Graduate Center, City University of New York}\email{jmeisel@gradcenter.cuny.edu}

\begin{document}

\begin{abstract}
In the bullet process, a gun fires bullets in the same direction at independent random speeds, and with independent random time delays between firings. When two bullets collide, they vanish. The critical velocity $v_c$ is the slowest speed the first bullet can take and still have positive probability of surviving forever. We characterize the critical velocity via a random variable determined by the sequence of speeds and delays, which we show almost surely equals $v_c$. In turn we prove other facts about the process, including that infinitely many bullets survive when the velocity distribution has finite support. Along the way we answer a question from Broutin--Marckert (2020), showing that if a bullet survives, it does so in all but finitely many truncations of the process.
\end{abstract}

\maketitle

\section{Introduction}

In the \emph{bullet process}, bullets $\b_i$ for $i \ge 0$ are fired from the origin. The $i^\text{th}$ bullet is fired at time $t_i = \Delta_1 + \cdots + \Delta_i$, so that $t_0 = 0$ and $\Delta_i = t_i - t_{i-1}$ for $i \ge 1$. Bullet $b_i$ travels along $[0,\infty)$ at constant velocity $v_i > 0$. Whenever a faster bullet catches up to a slower one, they both annihilate and are removed from the system. All velocities $v_i$ for $i \ge 0$ and delays $\Delta_i$ for $i \ge 1$ are independent, drawn from \emph{velocity distribution} $\mu$ and \emph{delay distribution} $\nu$, both supported on $(0, \infty)$. Figure \ref{fig:diagram} depicts a simulation of the process. In this paper, we assume $(\mu, \nu)$ are such that almost surely no three bullets simultaneously collide --- which holds for instance if either $\mu$ or $\nu$ is non-atomic --- and call such $(\mu, \nu)$ a \emph{valid pair} of distributions. We denote the probability that the first bullet $\b_0$ survives when $v_0$ is set to $v$ by $\theta(v) = \theta(v, \mu, \nu)$, with the \emph{critical velocity} $v_c = v_c(\mu, \nu) \le \infty$ defined as 
\[
    v_c := \inf\{v \ge 0 \colon \theta(v) > 0\}.
\]

\begin{figure}[htbp]
	\centering
	\includegraphics[width=0.65\linewidth]{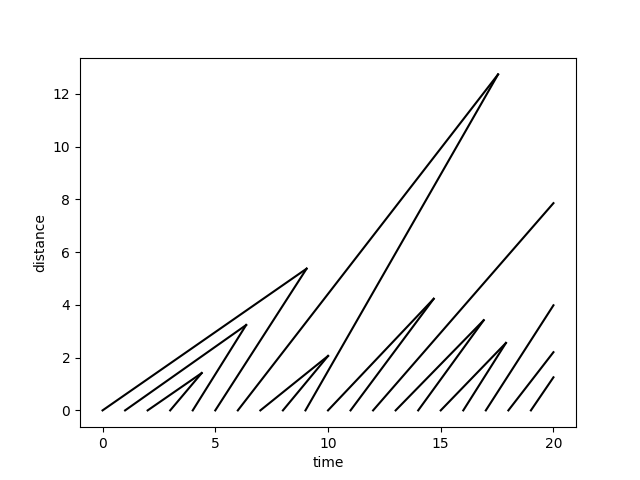}
	\caption{The time-space diagram for a sample path of the bullet process up to time $t=20$ in the setting with $\mu = \operatorname{Unif}(.5,1.5)$ and unit delays, i.e.\ $\nu = \delta_1$. Time moves left to right, with distance to the origin depicted vertically.}
	\label{fig:diagram}
\end{figure}

\begin{remark}
	We assume all speeds are strictly positive for convenience in the arguments. However, all our results go through if $\mu$ is supported on $[0, \infty)$ provided $\mu$ is not deterministically $0$ (when $\mu = \delta_0$ the bullet process is trivial to analyze, with each bullet $\b_{2k+1}$ for $k \ge 0$ annihilating $\b_{2k}$ at time $t_{2k+1}$). The first time a speed-$0$ bullet $\b_i$ is fired, it is immediately annihilated once $\b_{i+1}$ is fired. This has the net effect of simply extending the delay after the previous bullet from $\Delta_{i}$ to $\Delta_{i} + \Delta_{i+1} + \Delta_{i+2}$, and similarly with later speed-$0$ bullets. It is not hard then to couple this with a bullet process with strictly positive velocities and longer delays. 
\end{remark}

The bullet process was first defined by David Wilson \cite{wilson}, in the setting with velocity distribution $\mu = \operatorname{Unif}(0,1)$ and unit delays, i.e.\ $\nu = \delta_1$. The \textit{bullet problem} asks whether the first bullet survives with positive probability, that is $v_c(\operatorname{Unif}(0,1), \delta_1)$ is strictly less than $1$. Despite its simple setup, a solution has remained elusive over the last decade, contributing to the bullet problem's popularity. It is immediate that $v_c(\mu, \nu) \le \sup (\supp(\mu))$, since $b_0$ almost surely survives if we set $v_0 \ge \sup (\supp(\mu))$. It is conjectured that in the context of the bullet problem, $v_c \in (0,1)$, with our simulations placing it somewhere around $.75$ (see Figure \ref{fig:theta}). This updates the word-of-mouth conjecture that $v_c \sim 0.9$ from \cite{dygertJunge}.

\begin{figure}[htbp]
	\centering \includegraphics[width=0.75\linewidth]{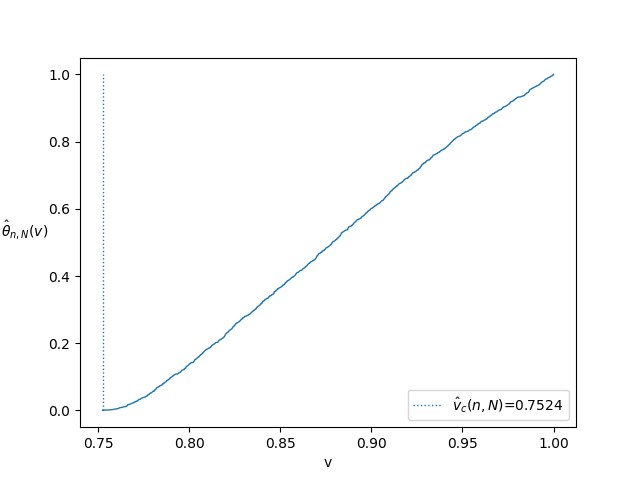}
	\caption{\textbf{\boldmath Approximation of $\theta$ and $v_c$.} We approximate $\theta(v)$ in the context of the bullet problem where $\mu = \operatorname{Unif}(0,1)$ and $\nu = \delta_1$, using $N=2000$ simulations of $n=1$ billion bullets. Letting $\theta_n(v)$ be the probability the first bullet survives with velocity $v$ among $n$ bullets, and $\hat{\theta}_{n,N}(v)$ its empirical estimation based on $N$ simulations, we find $\hat{v}_c(n, N) \sim .7524$, the minimal $v$ with $\hat{\theta}_{n,N}(v) > 0$. By Theorem \ref{thm:Marckert} we have $\theta_n(v) \to \theta(v)$ as $n \to \infty$, and therefore $\lim_{N \to \infty}\lim_{n \to \infty}\hat{v}_c(n, N) = v_c$ in probability.} 
	\label{fig:theta}
\end{figure}

In our main result, we define a random variable $\hat{v} = \hat{v}(\mu, \nu)$, measurable with respect to the velocities and delays, and prove it almost surely equals $v_c$. Thus the critical velocity, defined as a property of the entire distribution of the velocity and delay sequences, emerges pathwise as a property of almost every realization of the sequences. We define $\hat{v}$ in terms of the velocities of \emph{potential survivors}, bullets $\b_n$ which survive when only bullets $(\b_0, \ldots, \b_n)$ are fired. Note that $\b_n$ survives if and only if it is a potential survivor, and also survives among $(\b_n, \b_{n+1}, \ldots)$. Both events are independent conditional on $v_n$. With $(\b_{j_0}, \b_{j_1}, \ldots)$ all potential survivors ordered by increasing index, we let 
\[
    \hat{v} := \limsup_n v_{j_n}.
\]

Throughout, all claims involving random variables hold almost surely.

\begin{theorem} \label{thm:main}
	For any valid velocity and delay distributions $(\mu, \nu)$, 
    \[
        \hat{v}(\mu, \nu)= v_c(\mu, \nu).
    \]
\end{theorem}

Using Theorem~\ref{thm:main}, we can numerically approximate $v_c$ by approximating $\hat{v}$ for a single realization of the bullet process. See Figure \ref{fig:hat_v} for an alternative approximation of the critical velocity in the bullet problem.

\begin{figure*}[tbp]
	\centering
	\begin{subfigure}{0.5\textwidth}
		\centering
		\includegraphics[width=\linewidth]{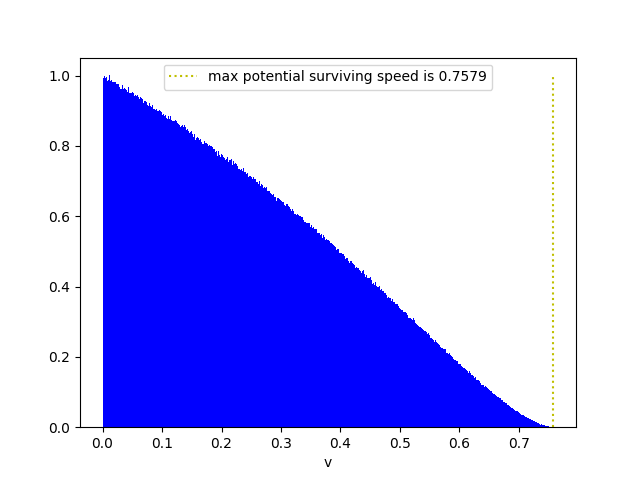}
	\end{subfigure}%
	~ 
	\begin{subfigure}{0.5\textwidth}
		\centering
		\includegraphics[width=\linewidth]{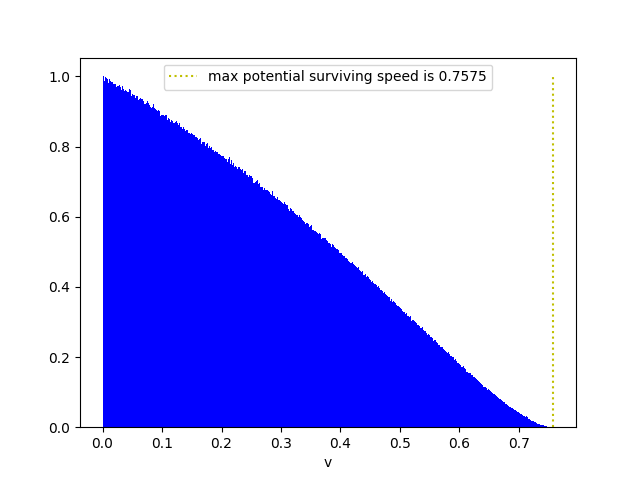}
	\end{subfigure}
	\caption{\textbf{\boldmath Approximation of $\hat{v} = v_c$.} We approximate $\hat{v}$ in two different simulations of the bullet process in its original setting using $n=10^8$ bullets each. We consider potential survivors fired between times $n/2+1$ and $n$. The yellow line displays their maximum velocity. A histogram of their velocities is shown in blue. We use buckets of width $.001$, and the height of each bar is the number of potential survivors with velocity in the bucket divided by $.001 (n/2)$, estimating the probability a velocity-$v$ bullet will be a potential survivor. While we have no guarantees on the rate of convergence for $\hat{v}$, the simulations appear to be stabilizing, giving a seemingly accurate estimate of $v_c$ in concordance with that from Figure~\ref{fig:theta} obtained by estimating $\theta$.}
	\label{fig:hat_v}
\end{figure*}

Moreover, Theorem~\ref{thm:main} gives us a handle on the bullet process which we leverage to analyze other properties. For instance, a question we sought to answer was whether infinitely many bullets survive. We prove this is the case when there are finitely many velocities, with all but finitely many survivors having critical velocity. 

\begin{theorem} \label{thm:finite}
	For any valid $(\mu, \nu)$ such that $\mu$ has finite support, almost surely infinitely many bullets survive, and all but finitely many survivors have velocity $v_c(\mu, \nu)$. 
\end{theorem}

See Figure \ref{fig:finite} for an illustration of this fact in a discretization of the bullet problem setting. For general velocity distributions, we were not able to answer whether infinitely many bullets survive, but we do demonstrate a zero-one law in Corollary \ref{cor:zero_one}.

We also use Theorem \ref{thm:main} to obtain continuity results for $\theta$, which by definition is non-decreasing in $v$.

\begin{theorem} \label{thm:cont}
	For any valid $(\mu, \nu)$,
	\begin{enumerate}
		\item $\theta$ is right-continuous at all $v \ge 0$.
		\item For any $v > v_c(\mu, \nu)$, $\theta$ is discontinuous at $v$ if and only if $v$ is an atom of $\mu$. 
		\item Let $E$ denote the event that there are infinitely many potential survivors faster than $v_c = v_c(\mu, \nu)$. The following are equivalent:
			\begin{enumerate}
				\item $\theta$ is continuous at $v_c$. 
				\item $\theta(v_c) = 0$.
				\item $\P(E) > 0$.
				\item $\P(E) = 1$. 
			\end{enumerate}
	\end{enumerate}
\end{theorem}

Finally, we note that all of our results apply to \textit{one-sided ballistic annihilation}. Ballistic annihilation is a closely related process studied by physicists to model the kinetics of chemical reactions \cite{ballistic}. All particles are present at time $0$ and have i.i.d.\ spacings. As noted by Sidoravicius and Tournier in \cite{sidoravicius}, by inverting the axes on the time-space diagram, one-sided ballistic annihilation is equivalent to the bullet process with inverted velocities. Using the \textit{linear speed-change invariance property} from Section 2 of \cite{sidoravicius}, all of our results apply to one-sided ballistic annihilation with velocity distribution bounded above or below. 

\subsection{Definitions and Notation}\label{sec:definitions} 

A bullet is a triple 
\[
    \b_i = (i, v_i, \Delta_i),
\]
with index $i \ge 0$, velocity $v_i \ge 0$, and delay $\Delta_i > 0$ (we allow $v_i = 0$ for our definition of $\theta$). Given a deterministic sequence of bullets $\vec{\b}_\mathcal{J} = (\b_i)_{i \in \mathcal{J}}$, indexed by the finite or countable set of consecutive nonnegative integers $\mathcal{J} = \{m, m+1, \ldots\} \subseteq \mathbb{N}_{\ge 0}$, we define the bullet process 
\[
    B(\vec{\b}_\mathcal{J}) = \big(a_t(i), d_t(i)\big)_{t \ge 0,\, i \in \mathcal{J}}
\]
as follows, where $a_t(i) \in \{0,1\}$ denotes whether $b_i$ is active and $d_t(i)$ is the virtual trajectory of $b_i$, giving its position at times when it is active and ignored when it is not. Each bullet $b_i$ for $i \in \mathcal{J}$ is fired at time $t_{i} = \sum_{m < j \le i}\Delta_{j}$ (note that the first delay $\Delta_m$ plays no role for this sequence of bullets). Then we let $d_t(i) := v_i(t - t_i)$ for all $t \ge 0$. We let $a_t(i) := 0$ for all $0 \le t < t_i$, and at time $t_i$ switch $a_t(i)$ to $1$. It remains the case that $a_t(i) = 1$ until the first time, if ever, $b_i$ collides with another active bullet $b_j$ --- say at time $s > t_i$ --- in which case we set $a_t(i) := 0$ for all $t > s$. For $i,j \in \mathcal{J}$ with $i < j$, we let $\b_j \ra \b_i$ denote the event that $\b_j$ catches up to $\b_i$ and annihilates it.

For valid distribution pair $(\mu, \nu)$, let $\omega_i = (v_i, \Delta_i)$ for $i \ge 0$ be drawn i.i.d.\ from $\mu \otimes \nu$, and $\omega = (\omega_0, \omega_1, \ldots)$. We denote the full bullet sequence by
\[
    \vec{\b} = \vec{\b}(\omega) := \left(b_i(\omega)\right)_{i \ge 0}.
\]
We denote the bullet process on the full sequence by 
\[
    B = B(\omega) := B(\vec{\b}).
\]
The law of $ B$ is shift-invariant: with shift operator $T$ defined by $T(\omega_0, \omega_1, \ldots) =  (\omega_1, \omega_2, \ldots)$, we have $\vec{\b} \circ T = \vec{\b}(T\omega) \overset{d}{=} \vec{\b}$, and thus $B \circ T \overset{d}{=} B$.

We obtain subsequences of $\vec{\b}$ via the following indexing conventions:
\begin{align*}
	\vec{\b}_{n} &= (\b_0, \ldots, \b_n)\\
	\vec{\b}_{[m,n]} &= (\b_m, \ldots, \b_n)\\
	\vec{\b}_{[n,\infty)} &= (\b_n, \b_{n+1}, \ldots)
\end{align*}
In particular, $\vec{\b}_{n} = \vec{\b}_{[0,n]}$. For index $I$ of the form $[m,n]$, $[n,\infty)$ or $n$, we denote the bullet process on $\vec{\b}_I$ by $B_I := B(\vec{\b}_I)$. To specify which bullet process a collision takes place in, we use phrases like $\b_j \ra \b_i$ \emph{in} $B_n$ or \emph{in} $B_{[n,\infty)}$. By itself, $\b_j \ra \b_i$ means it occurs in $B$, the process on the full sequence of bullets.

We say $\b_i$ \textit{perishes} if $\b_j \to \b_i$ or $\b_i \to \b_j$ for some $j$, otherwise we say it \textit{survives}. Bullet $\b_j$ \textit{threatens} $\b_i$ if $\b_j \ra \b_i$ in $B_j$, which is a precondition for $\b_j$ annihilating $\b_i$ in $B$. The set of surviving bullets in $B$ is denoted 
\[
    S = S(\vec{\b}) := \{\b_i : \b_i \text{ survives in } B(\vec{\b})\}.
\]
The set of survivors for $B_I$ is denoted $S_I = S(\vec{\b}_I)$. We call $\b_n$ a
potential survivor if it survives in $B_n$, i.e.\ $\b_n \in S_n$, and let $PS = PS(\vec{\b})$ denote the set of potential survivors in $B$. Similarly, $PS_I$ is the set of potential survivors in $B_I$, where for instance $PS_{[n,\infty)}$ is the set of all $\b_{n'}$ that survive in $B_{[n,n']}$.  Almost surely $|PS| = \infty$ since $\b_n \in PS$ whenever it is the slowest bullet among $\vec{\b}_{[0,n]}$. So letting $(\b_{j_0}, \b_{j_1}, \ldots)$ enumerate $PS$ by increasing index $j_n$, we can define $\hat{v} = \limsup_{n \to \infty}v_{j_n}.$

The bullet $\b_n$ \textit{survives behind} if it survives in $B_{[n,\infty)}$. Note that $\b_n \in S$ if and only if it is a potential survivor and survives behind, so $S \subseteq PS$. 

For any set of bullets $A$, let $A_{>v}$ denote $\{\b_i \in A\colon v_i > v\}$, and similarly define $A_{\ge v}, A_{=v}$. So for instance $(S_{[n,\infty)})_{>v}$ is the set of survivors of $B_{[n,\infty)}$ with velocity greater than $v$.

We define $\theta(v)$ for all $v \ge 0$, not just $v \in \supp \mu$, as follows. To refer to the setting in which we fix the speed of $b_0$ to $v$, we let $\b^v_0 := (0, v, \Delta_0)$, and let $\SurvFirst$ denote the event that $\b_0^v$ survives in $B(\b^v_0, \vec{\b}_{[1,\infty)})$. Then we define
\[
    \theta(v) := \P(\SurvFirst).
\]
If $\theta(v) > 0$, we say $v$ or a speed-$v$ bullet is \textit{fast} or \textit{can survive}, and otherwise is \textit{slow} or \textit{can't survive}. While, as discussed above, necessarily $v_c\le \sup(\supp( \mu))$, it does not follow immediately by definition that $v_c \in \supp \mu$, or even that $v_c \ge \inf (\supp (\mu))$, but this is indeed the case by Theorem \ref{thm:main}, as we will establish.

\begin{remark}\label{rem:shifts}
	Note that $S_{[n, \infty)} = S_{[n, \infty)}(\omega)$ and $S \circ T^{n} = S(T^{n}\omega)$ agree up to a reindexing, that is they are in bijection under the mapping $\b_i \mapsto (i - n, v_i, \Delta_i)$. Thus they contain the same velocities, and likewise for $PS_{[n, \infty)}$ and $PS \circ T^{n}$. We then get $|(S_{[n, \infty)})_{\ge v}| = |S_{\ge v} \circ T^{n}| \overset{d}{=} |S_{\ge v}|$, along with other equations of that form.
\end{remark}

\subsection{Application to the Bullet Problem} \label{sec:application}
In this section, we outline a possible avenue of attack for the bullet problem by considering potential survivors. Note that $\b_n$ threatens $\b^v_0$ if and only if both: \begin{enumerate}[label=(\roman*), ref=(\roman*)]
	\item \label{cond-i} Bullet $\b_n$ is in $PS(\b_1, \ldots, \b_n)_{>v}$.
	\item Bullets $(\b_1, \ldots, \b_{n-1})$ mutually annihilate in $B_{[1, n-1]}$, i.e.\ $|S_{[1, n-1]}|=0$.
\end{enumerate}
 In \cite{broutinMarckert} it was proved that $\P(|S_{[1, n-1]}|=0) = O(n^{-1/2})$. Suppose it could be shown that adding condition~\ref{cond-i} increases the exponent to something summable for some $v < 1$, so $\b_0^v$ is threatened by only finitely many bullets $\b_n$ a.s. Then for large enough $v < 1$, no bullet threatens $\b_0^v$ and it thus survives.

\subsection{Proof Overview} 

To analyze surviving bullets, we prove Lemma \ref{lem:threats}, the finite threats lemma, stating that a bullet $\b_i$ can survive only if it is threatened a finite number of times. The argument is that every bullet that threatens $\b_i$ has at least a $\theta(v_i)$ probability of surviving behind, thus annihilating $\b_{i}$. We demonstrate that this eventually happens using a kind of renewal argument that we take advantage of throughout the paper. 

With the finite threats lemma we answer a question posed by Broutin and Marckert in \cite{broutinMarckert}. They explained that $|S_n|$ and $|S_{n+1}|$ differ by $1$, and also proved the exact distribution of $|S_n|$. They posited that $|S| = 0$ when $|S_n| = 0$ infinitely often, reducing the bullet problem to a recurrence question. This follows from Theorem \ref{thm:Marckert}.

To demonstrate $\hat{v}=v_c$, we prove a slightly stronger statement.

\begin{theorem} \label{thm:stronger}
For any valid $(\mu, \nu)$, if $v$ is slow then $PS_{>v}$ is infinite. If $v$ is fast then $PS_{>v}$ is finite.
\end{theorem}

The first statement of the theorem implies that if $v < v_c$ then $\hat{v} \ge v$ a.s., and the second statement implies that if $v > v_c$ then $\hat{v} \le v$ a.s.

The outline of the proof of the first half of Theorem~\ref{thm:stronger} is as follows: take $v$ such that $PS_{>v}$ is finite with positive probability, which tells us that $\b_0^v$ may face finitely many threats, as $(PS_{[1, \infty)})_{> v}$ may be finite. Using this, we construct a shield $(\b_1, \ldots, \b_m)$ following $\b_0^v$, which ensures its survival with positive probability by protecting it from $(PS_{[m+1, \infty)})_{> v}$. 

At this point, we remark how to show $\hat{v}=v_c$ for the specific context of the bullet problem, providing a simpler argument than required for the general context (we do the same later for finite velocity distributions). Roughly, to show $PS_{>v}$ cannot be infinite for any $v > v_c$, we argue that each potential survivor could be followed by a survivor with velocity in $(v_c, v)$, so this would eventually happen.

The proof sketch for the second half of Theorem~\ref{thm:stronger} proceeds as follows. Taking a fast velocity $v$, we seek to derive a contradiction from $|PS_{>v}|=\infty$. First, we will show this implies $|S_{>v}|=\infty$ by a renewal argument. After obtaining a zero-one law on the infinitude of $|S_{>v}|$, this will contradict that $v$ is fast, since then $\b_0^v$ is always followed by infinitely many faster survivors. To achieve the zero-one law, we use Birkhoff's ergodic theorem to guarantee that $|(S_{[k,\infty)})_{>v}|=\infty$ for some $k$. Finally, we will show that $S_{[k,\infty)}$ and $S_{[k-1,\infty)}$ differ by at most one bullet, analogous to the same fact about $S_n$ and $S_{n+1}$. 

\section{Proofs} \label{sec:proofs} 

We begin by proving the finite threats lemma. 

\begin{lemma} \label{lem:threats}(Finite threats lemma).
	If $\b_i$ survives then it is threatened by only finitely many bullets. 
\end{lemma}

Note then that all bullets face finitely many threats: if $\b_i$ is annihilated, say at time $t$ in $B(\vec{\b})$, then $\b_n$ cannot threaten $\b_i$ whenever $t_n > t$. 

\begin{proof}
	It suffices to show that any bullet $\b_i$ either faces finitely many threats or does not survive. 
	
	For $\b_i$ to have positive probability of survival, we must have that $\b_i \in PS$ and $\theta(v_i) = \theta(v_i, \mu, \nu) > 0$, so assume this is the case. We construct an infinite sequence of stopping times $\tau_1 \le \tau_2 \le \ldots$ taking values in $\mathbb{N} \cup \{\infty\}$, such that $\b_{\tau_n}$ threatens $\b_i$ whenever $\tau_n < \infty$.  Let $\b_{\tau_1}$ be the first bullet that threatens $\b_i$ (if there is none then $\b_i$ faces finitely many threats and we are done). For any $\tau_n<\infty$, let $\b_{\tau_{n+1}}$ be the first bullet to threaten $\b_i$ fired after $\b_{\tau_n}$ is killed, and if no such $\b_{\tau_{n+1}}$ exists let $\tau_{n+1}=\infty$, and thus $\tau_{n+k} = \infty$ for all $k \ge 1$. We have 
	\begin{align*}
		\P(\tau_{n+1} < \infty \mid \tau_n < \infty) 
		&\le \P(\b_{\tau_n} \text{does not survive behind} \mid \tau_n < \infty) \\
		&= \E[1 - \theta(v_{\tau_n}) \mid  \tau_n < \infty] \\
		&< 1 - \theta(v_i),
	\end{align*}
	where for the last inequality we used that $\theta$ is non-decreasing in $v$, and that $v_{\tau_n} > v_i$ on $\tau_n < \infty$, since then $\b_{\tau_n}$ threatens $\b_i$. 
	
	Thus,
	\[
		\P(\tau_{n+1} < \infty) \le \big(1 - \theta(v_i)\big)^n \to 0.
	\]
	Almost surely then, there is a last $\tau_{n} < \infty$. Either $\b_{\tau_n}$ is not annihilated from behind, in which case it annihilates $\b_i$, or it is, after which time no bullet fired threatens $\b_i$.
\end{proof}

With the finite threats lemma we establish the following connection between $S_n$ and $S$, relating survivors among infinite bullets to finite truncations of the process. For a set $A$ of bullets, denote its characteristic function by $\chi_A$, with $\chi_A(\b_i) = 1$ for $\b_i \in A$ and $\chi_A(\b_j) = 0$ for $\b_j \notin A$.

\begin{theorem} \label{thm:Marckert}
	\leavevmode
	\begin{enumerate}
		\item $\chi_{S_n}\to \chi_S$ pointwise.
		\item $|S| = \underset{n \to \infty}{\liminf}\,|S_n|$.
	\end{enumerate}
\end{theorem}
\begin{proof}
	\leavevmode
	\begin{enumerate}
		\item\label{thm:Marckert-1}  If $\b_i \notin S$, it is annihilated at some time $t$, so $\b_i \notin S_n$ whenever $t_n > t$. Conversely, suppose $\b_i \in S$. Then $\b_i$ is a potential survivor, and by the finite threats lemma is threatened by only finitely many bullets $\b_r$. All the $\b_r$ are killed from behind by some time $t$. For any $n$ with $t_n \ge t$ then, we have $\b_i \in S_n$, since $\b_i$ can only be annihilated in $B_n$ by some $\b_r$ that threatens it, but all such $\b_r$ die from behind in $B_n$.
		
		\item\label{thm:Marckert-2} By Fatou's lemma and item~\ref{thm:Marckert-1}, we have $\liminf_{n \to \infty}|S_n| \ge |S|$. For the other direction, suppose $S = \{\b_{s_1}, \ldots, \b_{s_k}\}$ for $k < \infty$ with $s_1 < s_2 < \ldots s_k$. Then $S_{s_k} = S$. As in Figure~\ref{fig:marckert}, letting $l_1 = s_{k}+1$, and $r_1$ be the index of the bullet that kills $\b_{l_1}$, we again have $S_{r_1} = S$. With $l_2 = r_1+1$, and $\b_{r_2}$ killing $\b_{l_2}$, we have $S_{r_2} = S$. Repeating this, $S_n = S$ infinitely often, so $\liminf_{n \to \infty}|S_n|\le |S|$. 
	\end{enumerate}  
\end{proof}

\begin{figure}[tbp]
	\centering
	\includegraphics[width=0.75\linewidth]{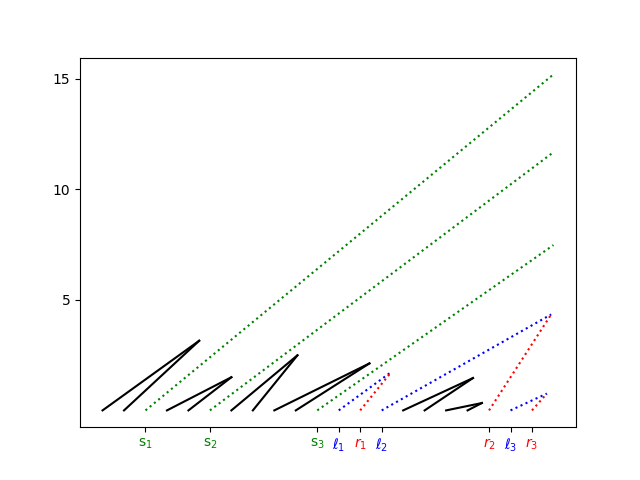}
	\caption{A bullet process with k=3 survivors, $S=\{\b_{s_1}, \b_{s_2}, \b_{s_3}\}$. We have $S_{r_n} = S$ for each $r_n$.}
	\label{fig:marckert}
\end{figure}

We can now prove half of Theorem 
\ref{thm:stronger}.

\Needspace{12\baselineskip}
\begin{lemma} \label{lem:stronger_half}
	If $v$ is slow then $|PS_{>v}| = \infty$.
\end{lemma}

\begin{corollary} \label{cor:vhat_ge_v_c}
	$\hat{v} \ge v_c$. 
\end{corollary} 

\begin{proof}[Proof of Corollary \ref{cor:vhat_ge_v_c}]
	For any $v < v_c$, $v$ is slow. By Lemma \ref{lem:stronger_half} then, there are infinitely many potential survivors faster than $v$, so $\hat{v} \ge v$. 
\end{proof}

\begin{proof}[Proof of Lemma \ref{lem:stronger_half}]
	Take $v \ge 0$ such that $E := \{|PS_{>v}| < \infty\}$ has positive probability. We show that $v$ is fast. 
	
	To do so, we construct a non-null event $\Shield$, measurable with respect to $(\omega_1, \ldots, \omega_{2m})$ for suitably chosen $m$, as well as a non-null event $F \subseteq E$, such that 
	\begin{equation}\tag{*}\label{eq:shield}
		\SurvFirst \supseteq \Shield \cap T^{-(2m+1)}F.
	\end{equation}
	 Figure \ref{fig:constr} depicts these events. Then since $\Shield$ and $T^{-(2m+1)}F$ are independent, the latter depending only on $(\omega_{2m+1}, \omega_{2m+2}, \ldots)$, we have
	\begin{align*}
		\theta(v) 
			&= \P(\SurvFirst)\\
			&\ge \P(\Shield \cap T^{-(2m+1)}F) \\
			&= \P(\Shield)\P(T^{-(2m+1)}F)\\
			&= \P(\Shield)\P(F) \\
			&> 0,
	\end{align*}
	and thus $v$ is fast. It remains to construct such non-null $\Shield$ and $F$ and demonstrate \eqref{eq:shield}.
	
	\begin{figure*}
		\centering
		\begin{subfigure}{\textwidth}
			\centering
			\includegraphics[width=.68\linewidth]{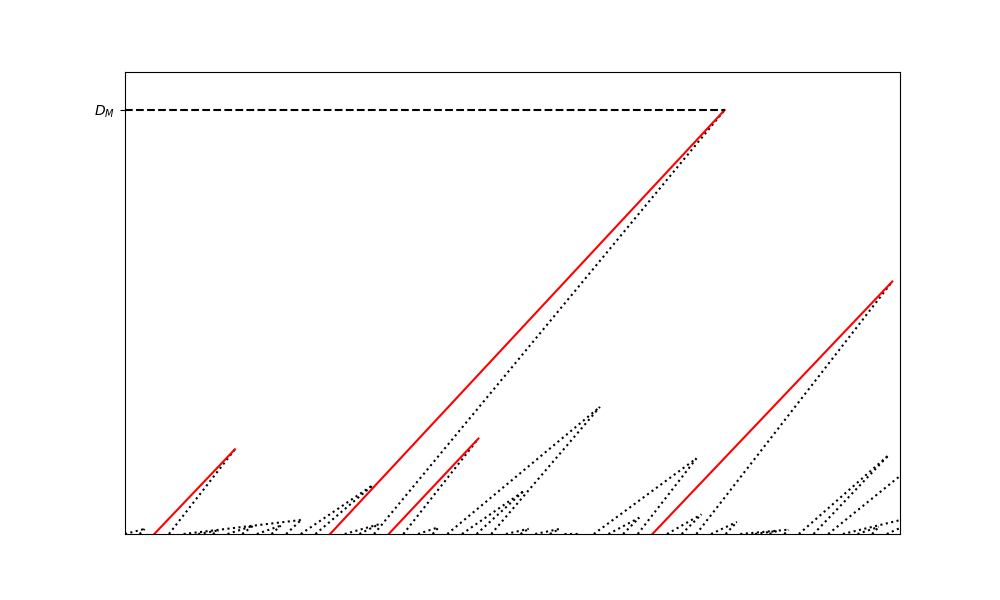}
		\end{subfigure}%
		
		\vspace*{-6mm}
		
		\begin{subfigure}{\textwidth}
			\centering
			\includegraphics[width=.68\linewidth]{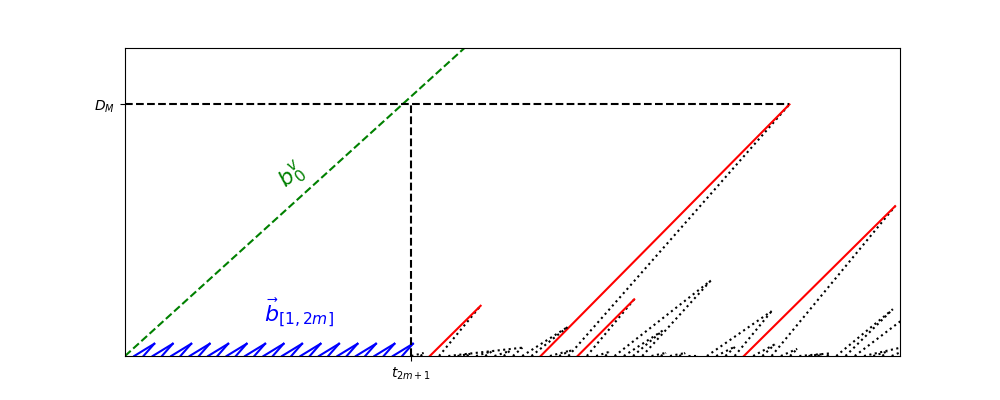}
		\end{subfigure}
		\caption{Above we depict $B(\omega)$ for $\omega \in F$. There are $4$ bullets in $PS_{>v}$, whose paths are drawn in red. Below, the path of $\b^v_0$ is green, and in blue is $\vec{\b}_{[1, 2m]}$, which shields $\b^v_0$ from its $4$ threats.}
		\label{fig:constr}
	\end{figure*}
	
	 We first show $E_0 := E \cap \{|S| = 0\}$ is non-null. To do so, we show that on $E$ there is always some random $k \ge 0$ such that $T^{-k} E_0$ holds. Therefore, $T^{-k} E_0$ must be non-null for at least some fixed $k \ge 0$, and then $E_0$ is non-null by shift invariance. So take $\omega \in E$, i.e.\ $\vec{\b}(\omega)$ is such that $|PS_{>v}| < \infty$.  If $S$ contains any bullets with velocity $\le v$, then we already know $v$ is fast, so assume that almost surely this is not the case. Then $S = S_{>v} \subseteq PS_{>v}$ is finite. Letting $\b_s$ be the last survivor in $S$, we show that $T^{-s} E_0$ holds, that is $|S_{[s+1,\infty)}|=0$ and $|(PS_{[s+1,\infty)})_{>v}| < \infty$ (see Remark~\ref{rem:shifts}). First off, $B_{[s+1,\infty)}$ has no survivors, as any such survivor would either survive in $B(\vec{\b})$ or annihilate $\b_s$, neither of which happens. To see $|(PS_{[s+1,\infty)})_{>v}| < \infty$, by Theorem~\ref{thm:Marckert} part~\ref{thm:Marckert-1}, we have $\b_s \in S_n$ for large $n$, at which point if $\b_n \in S_{[s+1,n]}$ — i.e.\ $\b_n \in PS_{[s+1,\infty)}$ — then $\b_n \in S_n$ — i.e.\ $\b_n \in PS$. Therefore, $PS_{[s+1,\infty)} \setminus PS$ is finite, so $(PS_{[s+1,\infty)})_{>v}$ is as well. This establishes that $E_0$ is non-null.
	
	We define $F$ as a subevent of $E_0$. For $\omega \in E_0$, let $D_M(\omega)$ be the furthest location any member of $PS_{>v}$ ever reaches, and let $V_M(\omega) > v$ denote the maximum velocity in $PS_{>v}$; we may assume $PS_{>v}$ is almost surely non-empty, as otherwise $v$ is fast since $\b_0^v$ is never threatened when $PS_{[1, \infty)}$ is empty. Choose $d_M,v_M$ so that with positive probability $d_M \ge D_M$ and $v_M \ge V_M$, moreover choosing $v_M$ so that $\mu\big([v_M,\infty)\big) > 0$. Let $F \subseteq E_0$ denote that these bounds on $D_M, V_M$ are satisfied. 
	
	To define $\Shield$, we choose $m$ large enough such that $t_{2m+1} > d_M / v$ with positive probability, and let $\Shield$ be the event that $t_{2m+1} > d_M / v$ and that for each $1 \le k \le m$, we have $v_{2k-1} \le v$ and $v_{2k} \ge v_M$. This event is non-null by choice of $v_M$ and $m$, and because $\mu\big((0,v]\big)$ must be positive since $PS_{>v}$ is possibly finite. It remains to show ~\eqref{eq:shield}.
	
	On $\Shield \cap T^{-(2m + 1)} F$, the truncated process $B(\b_0^v, \vec{\b}_{[1, 2m]})$ has collisions $\b_{2k} \ra \b_{2k-1}$ for $1 \le k \le m$ and $\b_0^v$ survives. These collisions still occur in the full process $B(\b_0^v, \vec{\b}_{[1, \infty)})$: each $\b_{2k}$ can only be threatened in the full process by some member of $(PS_{[2m+1, \infty)})_{>v_M}$, which is empty. Therefore, $\b_0^v$ can only be annihilated in the full process by a member of $(PS_{[2m+1, \infty)})_{>v}$, so this collision would have to take place at some location $x \le d_M$. But $\b^v_0$ reaches $d_M$ before any bullet in $\vec{\b}_{[2m+1, \infty)}$ is fired, and therefore survives.
\end{proof}

\begin{remark} \label{rem:canonical}
	We quickly state how to prove $\hat{v} = v_c$ in the context of the bullet problem. It remains to show $\hat{v} \le v_c$. Suppose by way of contradiction that $\hat{v} > v_c$. We force a survivor in $(v_c, \hat{v})$. Take $v,v'$ with $v_c < v < v' < \hat{v}$. Since $v$ is fast, $\theta(v) > 0$, and so each member of $PS_{>v'}$ has at least a $(v' - v)\theta(v) > 0$ chance of being followed by a survivor $\b_s$ with $v_s \in (v, v')$. Since $|PS_{>v'}|=\infty$, this eventually happens. (To rigorously show this, one can define an increasing sequence of indices $\tau_n$ in $PS_{>v'}$ where if $\b_{\tau_n}$ is followed by a bullet with velocity in $(v, v')$,  then $\b_{\tau_{n+1}}$ is fired after that bullet is annihilated). But then $\b_s \in S_n$ for large $n$, so $\b_n \notin PS_{>v'}$.
\end{remark}

Remark~\ref{rem:canonical} almost holds for general $\mu$. For any $v > v_c$ with $v \in \supp \mu$ we have $\hat{v} \le v$: if not, then with $N_v = (\frac{v_c + v}{2}, \frac{v+\hat{v}}{2})$, the above argument forces a survivor with speed in $N_v$ --- using that $\mu(N_v)\theta\big(\frac{v_c + v}{2}\big) > 0$ since $\frac{v_c + v}{2} > v_c$ is fast --- which contradicts that $\hat{v} > \frac{1}{2}(v_c + v)$. However, there is an obstacle if $\supp \mu$ has a gap above $v_c$. Take
 \[
	v_c^+ :=\inf \big(\supp \mu \cap (v_c, \infty)\big).
\]
Then $\hat{v}$ can only take values in $\{v_c, v_c^+\}$, but if $v_c^+$ is strictly larger than $v_c$, we still need to rule out the possibility that $\hat{v}$ can equal $v_c^+$. To overcome this obstacle, we go through Theorem \ref{thm:stronger}, starting with the following lemma.

\begin{lemma} \label{lem:inf_S}
	Let $v$ be a fast speed, and let $PS_{*v}$ be of the form $PS_{>v}$, $PS_{\ge v}$, or $PS_{= v}$. If $PS_{*v}$ is infinite, then so is $S_{*v} = PS_{*v} \cap S$. 
\end{lemma}
\begin{proof}
	Take any index $i$. We define an infinite sequence of stopping times $\tau_1 \le \tau_2 \le \ldots$, agnostic to whether $|PS_{*v}| = \infty$. Let $\tau_1 > i$ be minimal such that $\b_{\tau_1} \in PS_{*v}$, if any such $\tau_1 > i$ exists. Otherwise let $\tau_1 = \infty$. For $\tau_n < \infty$, if possible let $\b_{\tau_{n+1}}$ be the first bullet in $PS_{*v}$ fired after $\b_{\tau_n}$ perishes. If no such $\tau_{n+1}$ exists, that is $\b_{\tau_n} \in S_{*v}$ or there are not enough bullets in $PS_{*v}$, let $\tau_{n+1}=\infty$. Then $\P(\tau_{n+1} < \infty \mid \tau_{n} < \infty) \le 1-\theta(v)$, so a.s.\ $\tau_{n}=\infty$ for large $n$. Thus either $PS_{*v}$ is finite or there is some $\b_{\tau_n} \in S_{*v}$ with $\tau_n > i$.
\end{proof}

We obtain the following corollary. 

\begin{corollary} \label{cor:above_zero}
	$v_c \ge \inf (\supp (\mu))$.
\end{corollary}
\begin{proof}
	Take any $v < \inf(\supp(\mu))$, and suppose by way of contradiction that $v$ is fast. Then $PS_{>v}=PS$ is infinite, so $S_{>v}$ is as well by Lemma~\ref{lem:inf_S}. But then $\theta(v) = \P(\SurvFirst) = 0$, since $(S_{[1,\infty)})_{>v}$ is non-empty, being infinite. This contradicts that $v$ is fast.
\end{proof}

We now have the tools to prove there are infinitely many survivors when $\mu$ has finite support. 

\begin{proof}[Proof of Theorem \ref{thm:finite}]
	Suppose $\mu$ is supported on some finite set $V = \{V_1, \ldots, V_n\}$. Then $PS_{>\hat{v}}$ is always finite: since each potential survivor's velocity is in $V$, $\hat{v}$ is the maximal $V_m \in V$ which is the speed of infinitely many potential survivors. Thus $\hat{v}$ is fast by Lemma  \ref{lem:stronger_half}. Then since $PS_{=\hat{v}}$ is infinite, so is $S_{=\hat{v}}$ by Lemma \ref{lem:inf_S}. As the speeds of survivors are non-increasing, all but finitely many have speed $\hat{v}$.
	
	It remains to show $\hat{v} = v_c$. We could appeal to Theorem \ref{thm:main}, or alternatively, let $V_m$ be the minimal value $\hat{v}$ takes with positive probability. Again, $V_m$ is fast. Each member of $PS_{>V_m}$ is followed by a velocity-$V_m$ survivor with probability $\mu(\{V_m\})\,\theta(V_m)$. Arguing as in Remark \ref{rem:canonical} then, we see $PS_{>V_m}$ must be finite, so in fact $\hat{v} = V_m$. To see $V_m = v_c$, by the above, $|S_{=V_m}| = \infty$. But then for any $v < V_m$, we have $\theta(v) = \P(\SurvFirst) = 0$, since $B_{[1,\infty)}$ has infinitely many velocity-$V_m$ survivors.
\end{proof}

\begin{figure*}
	\centering
	\begin{subfigure}{0.5\textwidth}
		\includegraphics[width=.9\linewidth]{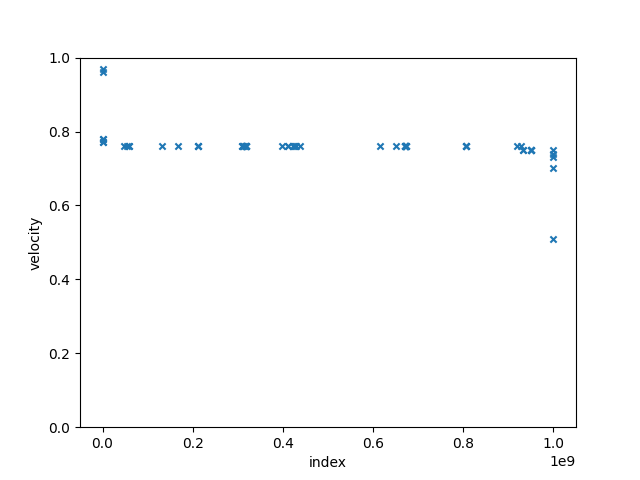}
	\end{subfigure}%
	~ 
	\begin{subfigure}{0.5\textwidth}
		\centering
		\includegraphics[width=.9\linewidth]{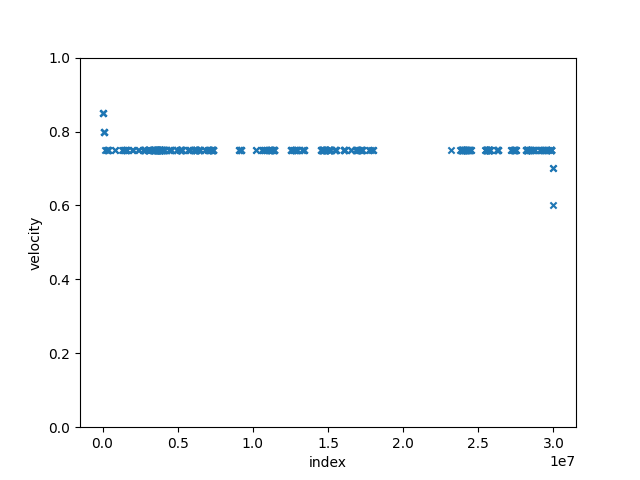}
	\end{subfigure}
	\caption{Two samples of $S_n$ --- the set of surviving bullets when $(\b_0, \ldots, \b_n)$ are fired --- for different uniform distributions on finite velocity sets. On the left, we use velocities $V_1, \ldots, V_{100}$ where $V_i = i/100 + \varepsilon_i$ and $\varepsilon_i \sim \mathcal{N}(0,.0002)$ is independent Gaussian noise added to avoid triple collisions. On the right we use the $20$ speeds $V_5, V_{10}, \ldots, V_{95}, V_{100}$, that is $.05, .1, \ldots, .95, 1.0$ plus noise. For the right, $n=30$ million and there are $574$ survivors. All but $12$ have velocity $V_{75}\sim.75$, with $9$ faster and $3$ slower at the end. On the left, we choose $n=1$ billion, as survivors are more spread out (Broutin and Marckert showed that for non-atomic velocity distributions, $|S_n|$ grows logarithmically \cite{broutinMarckert}). There are $62$ survivors: $52$ have velocity $V_{76}\sim \!\!.76$, $6$ are faster, and $4$ slower. The first speed-$V_{76}$ survivor has index $\sim48$ million.}
	\label{fig:finite}
\end{figure*}

\subsection{Zero-one law}

To prove a zero-one law for $|S_{*v}|=\infty$, we first analyze the effect adding a new bullet to the front has on the set of survivors. In Section 1.5 of \cite{broutinMarckert}, the case where a bullet is added to the back is characterized. Survivor sets $S_n$ and $S_{n+1}$ differ by exactly one bullet, with $S_{n+1}$ obtained by either removing the last survivor of $S_n$, or adding a new survivor to the back. We reproduce their argument from the front.

\Needspace{12\baselineskip}
\begin{lemma} \label{lem:delta_s}
	If $S_{[1,n]}$ is non-empty, call its first bullet $\b_{s_1}$. Then $S_n = S_{[0,n]}$ is either of the form $S_{[1,n]} \setminus \{\b_{s_1}\}$ or $S_{[1,n]} \cup \{\b_{s_0}\}$ with $s_0 < s_1$. If $S_{[1,n]}$ is empty then $|S_n|=1$. 
\end{lemma}

\begin{proof}
	Let $f=s_1 - 1$ when $S_{[1,n]}$ is non-empty and $f=n$ otherwise, so $\b_1, \ldots \b_f$ is the consecutive list of non-survivors at the front of $B_{[1,n]}$. There are three possibilities for $\b_0$ in $B_n$: it can survive, it can be annihilated by some $\b_\ell$ with $1 \le \ell \le f$ as in Figure \ref{fig:add_one}, or it can be annihilated by $\b_{s_1}$. If  $\b_0$ survives then $S_n = \{\b_{0}\} \cup S_{[1,n]}$, and if it is annihilated by $\b_{s_1}$ then $S_n = S_{[1,n]} \setminus \{\b_{s_1}\}$. The second scenario remains, when $\b_\ell \ra \b_0$ in $B_n$, whereas in $B_{[1,n]}$ we have $\b_r \ra \b_\ell$ for some $\ell < r \le f$. The collision between $\b_0$ and $\b_\ell$ frees up $\b_r$, so $S_n = S_{[r,n]}$ while $S_{[1,n]} = S_{[r + 1, n]}$. By strong induction we are done, since $B_{[r+1,n]}$ has $f-r < f$ non-survivors at the front.
\end{proof}
\begin{figure}[tbp]
	\centering
	\includegraphics[width=0.5\linewidth]{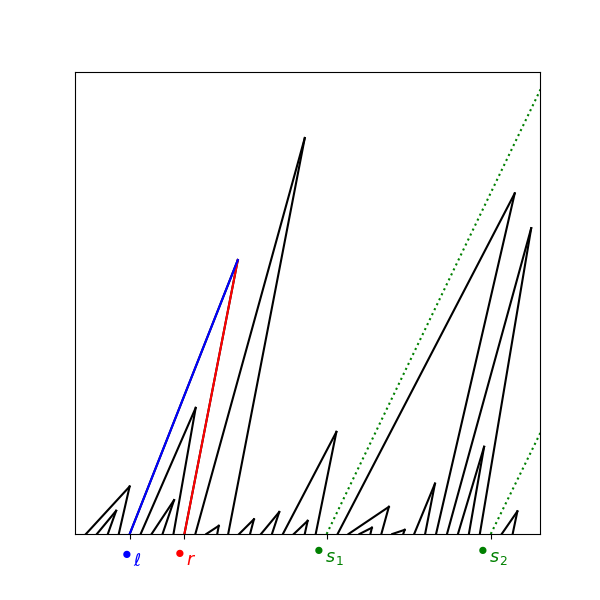}
	\caption{The process $B_{[1,n]}$ with two survivors, $\b_{s_1}$ and $\b_{s_2}$. The bullet $\b_\ell$ is one of the bullets that can potentially annihilate $\b_0$.}
	\label{fig:add_one}
\end{figure}

\begin{lemma} \label{lem:zero_one}
	For $S_{*v}$ of the form $S_{>v}$, $S_{\ge v}$, or $S_{=v}$, we have $\P(|S_{*v}| = \infty) \in \{0,1\}$. 
\end{lemma}
\begin{proof}
    Let $E = \{|S_{*v}| = \infty\}$, and suppose $\P(E) > 0$. We will show that $\P(E) = 1$. 

    By Birkhoff's ergodic theorem, 
    $$
        \lim_{n \to \infty}\frac{1}{n}\sum_{k=0}^{n-1} \mathbf{1}_E(T^k \omega) = \E[\mathbf{1}_E] = \P(E) > 0,
    $$
    using that the shift operator $T$ is an ergodic transformation for any i.i.d.\ sequence $\omega = (\omega_0, \omega_1, \ldots)$. Therefore, almost surely there are infinitely many $k$ such that $T^k \omega \in E$, and thus $|(S_{[k,\infty)})_{*v}| = \infty$ (see Remark~\ref{rem:shifts}). 

    Fix $k \ge 1$. We now show that on $|(S_{[k,\infty)})_{*v}| = \infty$, almost surely $|S_{[k,\infty)} \setminus S| \le k$, and thus $|S_{*v}| = \infty$ as well. Once this is established, it follows that $S_{*v}$ is almost surely infinite, so indeed $\P(E) = 1$. 
    
    Suppose then that $|(S_{[k,\infty)})_{*v}| = \infty$. Enumerate $S_{[k,\infty)} = \{\b_{s_0}, \b_{s_1}, \ldots\}$ with $k \le s_0 < s_1 < \ldots$. It suffices to show $\b_{s_i} \in S$ for all $i \ge k$. By Theorem~\ref{thm:Marckert} part~\ref{thm:Marckert-1}, we have $\{\b_{s_0}, \ldots \b_{s_i}\} \subseteq S_{[k,n]}$ for large $n$. Then $\{\b_{s_1}, \ldots \b_{s_i}\} \subseteq S_{[k-1,n]}$ by Lemma \ref{lem:delta_s}. Applying Lemma \ref{lem:delta_s} again we get $\{\b_{s_2}, \ldots \b_{s_i}\} \subseteq S_{[k-2,n]}$. Repeating this $k$ times, $\{\b_{s_{k}}, \ldots \b_{s_i}\} \subseteq S_{[0,n]} = S_n$ for large $n$. By Theorem \ref{thm:Marckert} again then, $\{\b_{s_{k}}, \ldots \b_{s_i}\} \subseteq S$. 
\end{proof}

\begin{corollary} \label{cor:zero_one}
	$\P(|S| = \infty) \in \{0,1\}$. 
\end{corollary}
\begin{proof}
	This is a direct consequence of Lemma~\ref{lem:zero_one} since $S = S_{\ge 0}$.
\end{proof}

\subsection{Remaining proofs}

We can now prove our two main theorems.

\begin{proof}[Proof of Theorem \ref{thm:stronger}]
	Lemma~\ref{lem:stronger_half} is the first half of the theorem. For the second half, suppose by way of contradiction that $v$ is fast and $\P(|PS_{>v}| = \infty) > 0$. By Lemmas \ref{lem:inf_S} and \ref{lem:zero_one} then, almost surely $S_{>v}$ is infinite. But then $v$ is slow, since $(S_{[1,\infty)})_{>v}$ is never empty.
\end{proof}

\begin{proof}[Proof of Theorem \ref{thm:main}]
	Any $v > v_c$ is fast, so $\hat{v} \le v$ almost surely by Theorem \ref{thm:stronger}, and therefore $\hat{v} \le v_c$ almost surely. Corollary \ref{cor:vhat_ge_v_c} is the other direction.
\end{proof}

Finally, we establish our continuity results.

\begin{proof}[Proof of Theorem \ref{thm:cont}]
	\leavevmode
	\begin{enumerate}
		\item Right-continuity of $\theta$ follows straightforwardly from the definition of the bullet process. Whenever $\SurvFirst$ fails, so does $\SurvFirst[u]$ for slightly faster $u$. 
		
		More precisely, for any $u > v$, we have 
		\begin{equation}\label{eq:theta-u-minus-theta-v}
			\theta(u) - \theta(v) = \P(\SurvFirst[u]) - \P(\SurvFirst) = \P(\SurvFirst[u] \setminus \SurvFirst),
		\end{equation}
		using that the event $\SurvFirst$ is increasing in $v$. Let $v_m(\omega) \in [0,\infty]$ be the infimum over $v' 
		\ge 0$ for which $\omega \in \SurvFirst[v']$, so 
		\begin{equation}\label{eq:vm}
			\SurvFirst[u]  \setminus \SurvFirst \subseteq \{v_m \le u\} \cap \SurvFirst^c.
		\end{equation}
		On $\SurvFirst^c$, bullet $\b_0^v$ collides in $B(\b_0^v, \vec{\b}_{[1, \infty)})$ with some faster $\b_r$ at some time $t$. Since triple collisions do not occur, as in Figure \ref{fig:cont} we still have $\b_r \ra \b_0^{v'}$ for slightly faster $v' > v$, since $\b_r$ survives to at least some $t' >  t-\Delta_1$ in $B_{[1,\infty)}$. Therefore, $v_m$ is strictly larger than $v$ almost surely. As $u \searrow v$ then, using \eqref{eq:theta-u-minus-theta-v} and \eqref{eq:vm},
		\[
			\theta(u) - \theta(v) \le \P(\SurvFirst^c, v_m \le u) \to 0.
		\]
		
		\begin{figure}[tbp]
			\centering
			\includegraphics[width=0.5\linewidth]{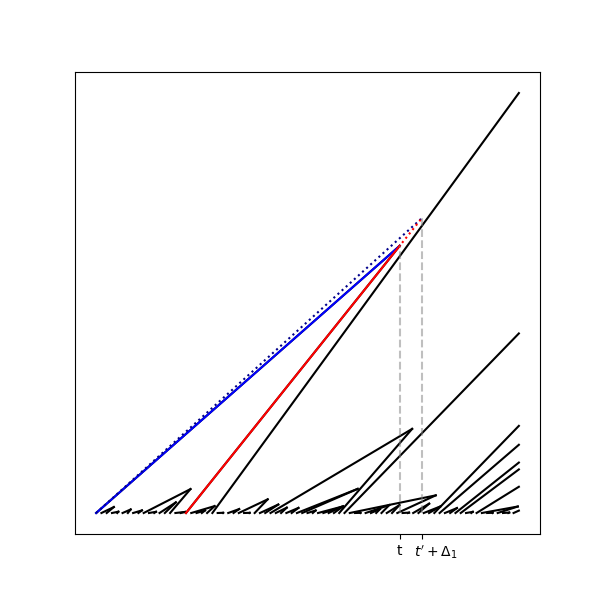}
			\caption{Bullet $\b_0^v$'s path is drawn in solid blue, and $\b_r$'s in red. Bullet $\b_0^{v'}$ is drawn in dotted blue. It still collides with $\b_r$ since there are no triple collisions.}
			\label{fig:cont}
		\end{figure}
		
		\item Take $v > v_c$. First suppose $\mu(\{v\}) > 0$. Let $E$ be the event that $v_1 = v$ and $\b_1 \in S_{[1,\infty)}$, so $\P(E) = \mu(\{v\})\theta(v) > 0$. As for any $u < v$, $E \subseteq \SurvFirst \setminus \SurvFirst[u]$, then by \eqref{eq:theta-u-minus-theta-v} we have $\theta(v) - \theta(u) \ge \P(E)$.
		
		Now suppose $v$ is not an atom of $\mu$. We need to show left-continuity. Take some $\omega \in \SurvFirst$. If $S_{[1,\infty)}$ is non-empty, let its first and fastest bullet have speed $v'$. Then $v'$ is not faster $v$, and is almost surely slower. If $S_{[1,\infty)} = \emptyset$, simply let $v'$ be the midpoint of $[v_c, v]$. As $v'$ is fast, $(PS_{[1,\infty)})_{>v'}$ is finite, and all its bullets die. Enumerate $(PS_{[1,\infty)})_{>v'} = (\b_{i_1}, \ldots \b_{i_n})$. These are the only bullets that could possibly threaten $\b_0^u$ for any $u \in [v',v]$. Now each $\b_{i_k}$ perishes in $B(\b_0^v, \vec{\b}_{[1,\infty)})$, and so is involved in some collision at time $t^k$ and location $x_k$. Bullet $\b^v_0$ avoids this collision by advancing to some $y_k > x_k$ at time $t^k$, as would $\b^{u_{k}}_1$ for a slightly slower $u_k < v$. Then $v_m(\omega)$ is at most the maximum of these $u_k$. Most importantly, $v_m(\omega) < v$ a.s. So for $u < v$, $\theta(v) - \theta(u) \le \P(\SurvFirst, v_m \ge u) \to 0$ as $u \nearrow v$. 
		
		\item With right-continuity established, continuity at $v_c$ is equivalent to $\theta(v_c)=0$, so the result follows from Theorem \ref{thm:stronger}. 
	\end{enumerate}
\end{proof}

\section*{Acknowledgements}
I would like to thank Matt Junge for his guidance and many stimulating conversations.

\bibliographystyle{amsalpha}
\bibliography{BA.bib}

\end{document}